\documentclass[a4paper, 11pt]{scrartcl}
\pdfoutput=1

\title{Mimetic Properties of Difference Operators: Product and Chain Rules as for
       Functions of Bounded Variation and Entropy Stability of Second Derivatives}
\author{Hendrik Ranocha}
\date{7th September 2018}
\titlehead{\includegraphics[width=0.5\textwidth]{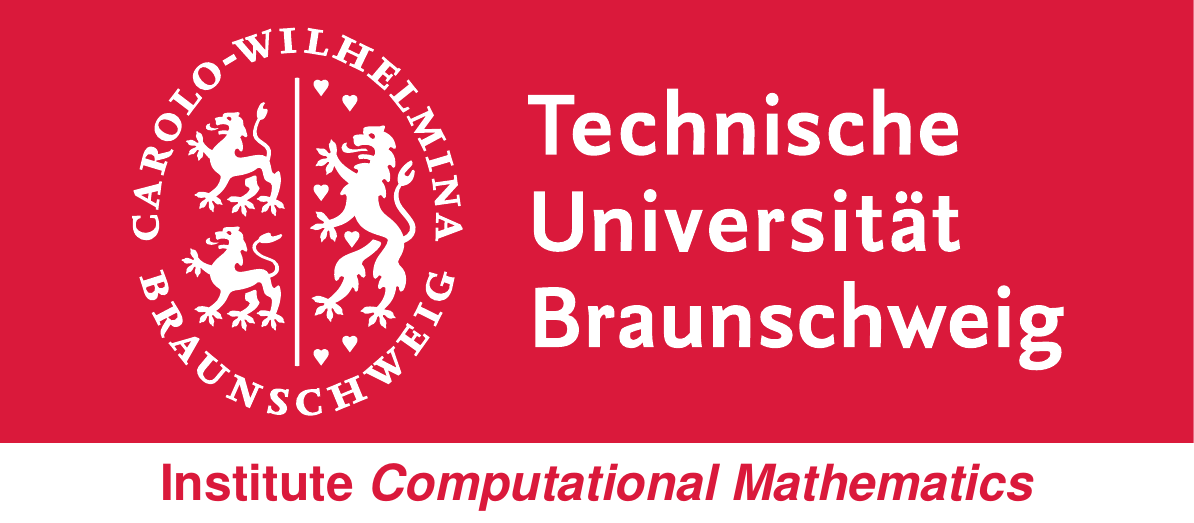}}

\usepackage[utf8]{luainputenc}
\usepackage[english]{babel}
\usepackage{csquotes}

\usepackage[a4paper, top=0.51in, bottom=0.79in, left=0.999in, right=0.99in]{geometry}

\usepackage[plainpages=false,pdfpagelabels,hidelinks,unicode]{hyperref}
\makeatletter
\hypersetup{pdfauthor={\@author}}
\hypersetup{pdftitle={\@title}}
\makeatother

\usepackage[%
  backend=biber,
  style=numeric-comp,
  giveninits=true, uniquename=init, 
  natbib=true,
  url=true,
  doi=true,
  isbn=false,
  backref=false,
  maxnames=99,
  ]{biblatex}
\usepackage{amsmath}
\allowdisplaybreaks
\usepackage{amssymb}
\usepackage{commath}
\usepackage{mathtools}
\usepackage{bbm}

\usepackage{siunitx}

\usepackage{amsthm}
\usepackage{thmtools}
\usepackage{etoolbox}
\makeatletter
\patchcmd{\thmt@setheadstyle}
 {\bgroup\thmt@space}
 {\thmt@space}
 {}{}
\patchcmd{\thmt@setheadstyle}
 {\egroup\fi}
 {\fi}
 {}{}
\makeatother
\declaretheoremstyle[
  bodyfont=\normalfont\itshape,
  headformat=\NAME\ \NUMBER\NOTE,
]{myplain}
\declaretheoremstyle[
  headformat=\NAME\ \NUMBER\NOTE,
]{mydefinition}
\newcommand{\envqed}{{\lower-0.3ex\hbox{$\triangleleft$}}}
\declaretheorem[style=myplain,numberwithin=section]{theorem}
\declaretheorem[style=myplain,numberlike=theorem]{lemma}

\declaretheorem[style=mydefinition,numberlike=theorem,qed=\envqed]{remark}
\declaretheorem[style=mydefinition,numberlike=theorem,qed=\envqed]{example}

\usepackage{ifluatex}
\ifluatex
  \usepackage[no-math]{fontspec}
\else
  \usepackage[T1]{fontenc}
\fi
\usepackage{newpxtext,newpxmath}

\usepackage{color}
\usepackage{graphicx}
\usepackage[small]{caption}
\usepackage{subcaption}

\usepackage{pgfplots}
\usepackage{tikzscale}
\pgfplotsset{compat=1.13}
\usepgfplotslibrary{groupplots}
\usepgfplotslibrary{external}
\begingroup\expandafter\expandafter\expandafter\endgroup
\expandafter\ifx\csname pdfsuppresswarningpagegroup\endcsname\relax
\else
\fi

\usepackage{booktabs}
\usepackage{rotating}
\usepackage{multirow}

\usepackage{multicol}
\usepackage{enumitem}

\usepackage{calc}
\usepackage{xparse}

\renewcommand{\vec}[1]{\mathbf{#1}}
\newcommand{\diag}[1]{\operatorname{diag}\left(#1\right)}

\AtBeginDocument{%
  
}

\DeclarePairedDelimiterX\newset[1]\lbrace\rbrace{\setaux #1||\endsetaux}
\def\setaux#1|#2|#3\endsetaux{\if\relax\detokenize{#2}\relax #1 \else #1 \;\delimsize\vert\; #2 \fi}
\renewcommand{\set}[1]{\newset*{#1}}

\newcommand{\id}{\operatorname{id}}

\newcommand{\BV}{\operatorname{BV}}
\newcommand{\BVloc}{\operatorname{BV}_\mathrm{loc}}

\renewcommand{\O}{\mathcal{O}}

\NewDocumentCommand{\ENO}{o}{%
  \IfValueTF{#1}{%
    ENO$(#1)$
  }{%
    ENO
  }%
}
\NewDocumentCommand{\mENO}{o}{%
  \IfValueTF{#1}{%
    mENO$(#1)$
  }{%
    mENO
  }%
}

\AtBeginDocument{%
  \let\rho\varrho
  \let\phi\varphi
  \let\epsilon\varepsilon
}
\newcommand{\N}{\mathbb{N}}
\newcommand{\R}{\mathbb{R}}

\newsavebox{\DelimiterBox}
\newlength{\DelimiterHeight}
\newlength{\DelimiterDepth}
\newsavebox{\ArgumentBox}
\newlength{\ArgumentHeight}
\newlength{\ArgumentDepth}
\newlength{\ResizedDelimiterHeight}
\newlength{\ResizedDelimiterDepth}

\ifluatex

\else

\fi

\newcommand{\QED}{}

\begin{document}

\maketitle

\begin{abstract}
  
For discretisations of hyperbolic conservation laws, mimicking properties of
operators or solutions at the continuous (differential equation) level discretely
has resulted in several successful methods. While well-posedness for nonlinear
systems in several space dimensions is an open problem, mimetic properties such
as summation-by-parts as discrete analogue of integration-by-parts allow a direct
transfer of some results and their proofs, e.g. stability for linear systems.

In this article, discrete analogues of the generalised product and chain rules
that apply to functions of bounded variation are considered. It is shown that
such analogues hold for certain second order operators but are not possible for
higher order approximations. Furthermore, entropy dissipation by second derivatives
with varying coefficients is investigated, showing again the far stronger mimetic
properties of second order approximations compared to higher order ones.

\end{abstract}

\section{Introduction}
\label{sec:introduction}

Ever since the widespread application of computers in numerical mathematics and
even before, finite difference methods have been successfully applied to
differential equations. An important task is the
development and investigation of stable and well-behaved numerical methods.
While some general purpose methods can give satisfying results under certain
circumstances, schemes that have been developed specifically for the target
equation can be advantageous, e.g. if some properties of operators or solutions
at the continuous level are mimicked discretely. This has been the goal of, e.g.,
geometric numerical integration methods for ordinary differential equations, cf.
\cite{hairer2003geometric,hairer2006geometric}.

In this regard, a well-developed theory of the problem at the continuous
(differential equation) level is very important since it can be used as a
guideline for the development of (semi-) discretisations. For linear systems
of hyperbolic conservation laws, energy estimates play a fundamental role in the
analysis of well-posedness \cite{gustafsson2013time}. An important technique is
integration-by-parts. Thus, summation-by-parts (SBP) as a discrete analogue has
been very successful, since manipulations at the continuous level can be mimicked
discretely, yielding stability and conservation results, cf.
\cite{kreiss1974finite,strand1994summation,carpenter1994time,carpenter1999stable}.
Further references and results can be found in the review articles
\cite{svard2014review,fernandez2014review}.

Considering nonlinear equations such as scalar conservation laws, functions of
locally bounded variation play an important role. Solutions of quasilinear
equations can become discontinuous in finite time, even if smooth initial data
and coefficients are given \cite{dafermos2010hyperbolic}. In his seminal work
\cite{volpert1967spaces}, Vol'pert investigated functions of locally bounded
variation and their application to conservation laws. Since such functions can
be discontinuous, he developed a corresponding notion of derivatives as measures.
Moreover, he investigated products and compositions of functions of locally
bounded variation and developed corresponding product and chain rules.
Furthermore, the concept of bounded variation is important for the analysis of
numerical methods for conservation laws since it implies compactness properties
(Helly's theorem), cf. \cite{harten1987nonlinearity}.

The investigation of semidiscretisations satisfying a single entropy inequality
has received much interest, cf.
\cite{tadmor1987numerical,tadmor2003entropy,lefloch2002fully,sjogreen2010skew,fjordholm2012arbitrarily,fisher2013high,gassner2016well,wintermeyer2017entropy,gassner2016split,ranocha2017shallow,ranocha2017comparison,ranocha2018generalised}.
For some conservation laws such as Burgers' equations, conservative corrections
to the product rule can be used to obtain $L^2$ dissipative schemes, cf.
\cite{gassner2013skew,ranocha2016summation,ranocha2017extended}. Therefore, it
is interesting whether the chain and product rules for functions of bounded
variation have discrete analogues.

Furthermore, the investigation of numerical dissipation operators has received
much interest, cf.
\cite{vonneumann1950method,mattsson2004stable,svard2009shock,ranocha2018stability}.
Such operators can be motivated by the vanishing viscosity approach to conservation
laws, cf. \cite{bianchini2005vanishing}. For general entropies, the investigation
of dissipation induced by such terms relies on the chain rule, cf.
\cite[Proof of Theorem~I.3.4]{lefloch2002hyperbolic}. Thus, it is
natural to investigate the entropy dissipation of difference approximations.

This article is structured as follows. At first, functions of bounded variation
are briefly reviewed in section~\ref{sec:BV}, focusing on the chain and product
rules. Next, corresponding difference operators are investigated in
section~\ref{sec:difference-operators}. It is proven that there are analogous
product and chain rules for classical second order periodic and SBP operators
(Lemma~\ref{lem:product-rule} and Lemma~\ref{lem:chain-rule}). Furthermore, it is
proven that such analogues do not exist for higher order difference approximations
of the first derivative (Theorem~\ref{thm:product-rule} and Theorem~\ref{thm:chain-rule}).
Thereafter, dissipation operators approximating second derivatives with possibly
varying coefficients are investigated in section~\ref{sec:laplace}. It is proven
that certain second order difference operators are dissipative for every convex
entropy (Theorem~\ref{thm:laplace-2nd-order}). Moreover, it is shown that such
a result is impossible for discrete derivative operators with higher order of accuracy
(Theorem~\ref{thm:laplace-higher-order}). Finally, a summary and discussion is given
in section~\ref{sec:summary}.

\section{Functions of Bounded Variation}
\label{sec:BV}

Functions of locally bounded variation, i.e. those locally integrable functions
whose distributional first derivatives are Radon measures, play an important role
in analysis, for example in the theory of scalar conservation laws as described
in the seminal work of Vol'pert \cite{volpert1967spaces}. Further results about conservation
laws and references can be found in the monograph \cite{dafermos2010hyperbolic},
e.g. Theorem~6.2.6 and chapter~XI. Some general results about functions of bounded
variation can be found in \cite{volpert1985analysis,evans2015measure}.

For functions of locally bounded variation, a product of a possibly discontinuous
function and a measure occurs in both the chain rule and the product rule. If the
function is integrable with respect to the measure, this product is well-defined
as a measure, cf. \cite{volpert1967spaces}. In one space dimension, a function of
bounded variation is continuous almost everywhere and the limits from the left
and the right exist everywhere. If $u \in \BVloc([a,b]; \R^m)$ and $g\colon \R^m \to \R$
is (for simplicity) continuous, then Vol'pert \cite{volpert1967spaces} defined
the averaged composition of $g$ and $u$ via
\begin{equation}
  \widehat{g(u)}(x)
  :=
  \int_0^1 g\bigl( u_- + s (u_+ - u_-) \bigr) \dif s,
\end{equation}
where $u_\pm = \lim_{\epsilon \searrow 0} u(x \pm \epsilon)$ are the unique limits
of $u$ from the left and right hand side, respectively.
With this definition, the following chain and product rules have been obtained in
\cite[Section~13]{volpert1967spaces}.
\begin{theorem}
  If $u \in \BV([a,b]; \R^m)$ and $f \in C^1(\R^m; \R)$, the averaged composition
  $\widehat{\partial_{u_k} f(u)}$ is locally integrable with respect to the measure
  $\partial_x u_k$ for $k \in \set{1,\dots,m}$, $f(u) \in \BVloc$, and
  \begin{equation}
  \label{eq:BV-chain-rule}
    \partial_x f(u)
    =
    \sum_{k=1}^m \widehat{\partial_{u_k} f(u)} \, \partial_x u_k.
  \end{equation}
  In particular, for $\mathfrak{u}, \mathfrak{v} \in \BV[a,b]$,
  \begin{equation}
  \label{eq:BV-product-rule}
    \partial_x (\mathfrak{u} \mathfrak{v})
    =
    \widehat{\mathfrak{u}} \partial_x \mathfrak{v}
    + \widehat{\mathfrak{v}} \partial_x \mathfrak{u}.
  \end{equation}
\end{theorem}

\begin{remark}
\label{rem:scalar-vs-vector-valued}
  Sometimes, it might be useful to distinguish vector valued functions
  $u \in \BV([a,b]; \R^m)$, $m \geq 2$, and scalar valued functions
  $\mathfrak{u} \in \BV([a,b]; \R^1)$ explicitly. In this case, a fracture font
  will be used for scalar valued functions. Nevertheless, the case $m = 1$ is not
  excluded for vector valued functions $u \in \BV([a,b]; \R^m)$ if not stated
  otherwise. If it is clear from the context whether a function is scalar valued
  or may be vector valued, the usual font will be used for simplicity.
\end{remark}

\begin{remark}
\label{rem:CR-as-PR}
  If $\mathfrak{u}, \mathfrak{v} \in \BV[a,b]$, define $u \in \BV([a,b]; \R^2)$
  by $u(x) = (\mathfrak{u}(x), \mathfrak{v}(x))$. Considering the function
  $f \in C^1(\R^2; \R)$, given by $f(u) = f(u_1, u_2) = u_1 u_2$, the chain rule
  \eqref{eq:BV-chain-rule} becomes
  \begin{equation}
    \partial_x (\mathfrak{u} \mathfrak{v})
    =
    \partial_x f(u)
    =
    \sum_{k=1}^2 \widehat{\partial_{u_k} f(u)} \, \partial_x u_k
    =
    \widehat{\mathfrak{v}} \partial_x \mathfrak{u}
    + \widehat{\mathfrak{u}} \partial_x \mathfrak{v}.
  \end{equation}
  Thus, the product rule \eqref{eq:BV-product-rule} is indeed a special case
  of the chain rule \eqref{eq:BV-chain-rule}.
\end{remark}

The product rule \eqref{eq:BV-product-rule} is also proven in the monograph
\cite[Section~6.4]{volpert1985analysis}. A generalisation of the corresponding
definition of a possibly nonconservative product $f(u) \partial_x v$ has been
developed and investigated by Dal~Maso, LeFloch, and Murat
\cite{dalmaso1995definition}. See also \cite{raymond1996new,lefloch1999representation}
for further studies.

In order to illustrate the general theory described above and lay some foundations
for the following comparison with discrete derivative operators, two examples
using jump functions will be considered. If $u, v \in \BV[a,b]$ are (scalar
valued) jump functions,
\begin{equation}
  u(x)
  =
  \begin{cases}
    u_-, & x < 0,
    \\
    u_+, & x > 0,
  \end{cases}
  \qquad
  v(x)
  =
  \begin{cases}
    v_-, & x < 0,
    \\
    v_+, & x > 0,
  \end{cases}
\end{equation}
they are of bounded variation and their derivatives are Radon measures. In particular,
the derivatives $\partial_x u$ and $\partial_x v$ are multiples of the Dirac measure
centred at zero. Viewing such a measure as a function mapping (measurable) sets to
real numbers,
\begin{equation}
  (\partial_x u)(A)
  =
  \begin{cases}
    u_+ - u_-, & \text{if } 0 \in A,
    \\
    0, & \text{else}.
  \end{cases}
\end{equation}
Thus, the product rule \eqref{eq:BV-product-rule} becomes in this case
\begin{equation}
\label{eq:BV-product-rule-step-function}
\begin{aligned}
  \bigl( \partial_x (u v) \bigr)\bigl( \set{0} \bigr)
  &=
  u_+ v_+ - u_- v_-
  \\
  &=
  \frac{u_+ + u_-}{2} (v_+ - v_-) + \frac{v_+ + v_-}{2} (u_+ - u_-)
  \\
  &=
  \bigl( \widehat{u} \partial_x v \bigr)\bigl( \set{0} \bigr)
  + \bigl( \widehat{v} \partial_x u \bigr)\bigl( \set{0} \bigr),
\end{aligned}
\end{equation}
where the measures on both sides of \eqref{eq:BV-product-rule} have been applied
to the set $\set{0}$ containing only the jump point. Similarly, for $f \in C^1$,
the chain rule \eqref{eq:BV-chain-rule} becomes
\begin{multline}
\label{eq:BV-chain-rule-step-function}
  \bigl( \partial_x f(u) \bigr)\bigl(\! \set{0} \!\bigr)
  =
  f(u_+) - f(u_-)
  =
  \int_0^1 f'\bigl( u_- + s (u_+ - u_-) \bigr) \dif s \,\cdot (u_+ - u_-)
  \\
  =
  \bigl( \widehat{f'(u)} \cdot \partial_x u \bigr)\bigl(\! \set{0} \!\bigr).
\end{multline}
Of course, the intermediate steps of \eqref{eq:BV-product-rule-step-function} and
\eqref{eq:BV-chain-rule-step-function} can also be seen as proofs of the general
product and chain rule in this special case.

Interpreting the difference $u_+ - u_-$ as a discrete derivative,
\eqref{eq:BV-product-rule-step-function} is a discrete
product rule and \eqref{eq:BV-chain-rule-step-function} is a discrete chain rule.
Both use
averages instead of the usual point values occurring in the continuous analogues
for differentiable functions. Thus, it is interesting whether this can be
generalised.

\section{Difference Operators}
\label{sec:difference-operators}

Consider a general discrete derivative/difference operator $D$, acting on grid
functions $\vec{u} = (\vec{u}_i)_i = \bigl( u(x_i) \bigr)_i$ defined on a possibly
non-uniform grid with nodes $x_i \in \R$ and $h := \min_i (x_{i+1} - x_i) > 0$.
Note that this includes both classical finite difference operators and spectral
collocation operators such as nodal discontinuous Galerkin ones.

In practice, the grid function is represented by the vector of its point values
and the discrete derivative operator by a matrix with entries $D_{ij}$. General
nonlinear operations such as composition or multiplication are conducted pointwise,
i.e. if $\vec{u}$ and $\vec{v}$ are two grid functions, their product $\vec{uv}$
is the grid function with components $(\vec{uv})_i = \vec{u}_i \vec{v}_i$.

\subsection{Classical Second Order Derivative Operators}

The classical second order finite difference operator on a uniform grid is given by
\begin{equation}
  (D \vec{u})_i = \frac{\vec{u}_{i+1} - \vec{u}_{i-1}}{2 h} \approx u'(x_i).
\end{equation}
The corresponding summation-by-parts (SBP) operator uses this stencil in the
interior and --- if the nodes $x_0, \dots, x_N$ are used --- the boundary closures
\begin{equation}
  (D \vec{u})_0 = \frac{\vec{u}_1 - \vec{u}_0}{h} \approx u'(x_0),
  \qquad
  (D \vec{u})_N = \frac{\vec{u}_N - \vec{u}_{N-1}}{h} \approx u'(x_N).
\end{equation}

Analogously to the product rule \eqref{eq:BV-product-rule-step-function} for
a step function of bounded variation, considering scalar valued grid functions
$\vec{u}$ and $\vec{v}$,
\begin{equation}
\begin{aligned}
  (D \vec{u} \vec{v})_i
  &=
  \frac{\vec{u}_{i+1} \vec{v}_{i+1} - \vec{u}_{i-1} \vec{v}_{i-1}}{2 h}
  \\
  &=
  \frac{\vec{u}_{i+1} + \vec{u}_{i-1}}{2} \frac{\vec{v}_{i+1} - \vec{v}_{i-1}}{2 h}
  + \frac{\vec{u}_{i+1} - \vec{u}_{i-1}}{2 h} \frac{\vec{v}_{i+1} + \vec{v}_{i-1}}{2}
  \\
  &=
  (A \vec{u})_i (D \vec{v})_i + (D \vec{u})_i (A \vec{v})_i,
\end{aligned}
\end{equation}
if the averaging operator $A$ is defined by
\begin{equation}
  (A \vec{u})_i = \frac{\vec{u}_{i+1} + \vec{u}_{i-1}}{2} \approx u(x_i).
\end{equation}
For the corresponding SBP operator, the terms at the left boundary are
\begin{equation}
\begin{aligned}
  (D \vec{u} \vec{v})_0
  &=
  \frac{\vec{u}_{1} \vec{v}_{1} - \vec{u}_{0} \vec{v}_{0}}{h}
  \\
  &=
  \frac{\vec{u}_{1} + \vec{u}_{0}}{2} \frac{\vec{v}_{1} - \vec{v}_{0}}{h}
  + \frac{\vec{u}_{1} - \vec{u}_{0}}{h} \frac{\vec{v}_{1} + \vec{v}_{0}}{2}
  =
  (A \vec{u})_0 (D \vec{v})_0 + (D \vec{u})_0 (A \vec{v})_0,
\end{aligned}
\end{equation}
if the boundary closures of $A$ are given by
\begin{equation}
  (A \vec{u})_0 = \frac{\vec{u}_1 + \vec{u}_0}{2} \approx u(x_0),
  \qquad
  (A \vec{u})_N = \frac{\vec{u}_N + \vec{u}_{N-1}}{2} \approx u(x_N).
\end{equation}
The terms at the right boundary are similar. This is summed up in
\begin{lemma}
\label{lem:product-rule}
  The classical second order derivative operator $D$ (on a periodic grid or with
  boundary closures given above) fulfils the product rule
  \begin{equation}
    D (\vec{u} \vec{v}) = (A \vec{u}) (D \vec{v}) + (D \vec{u} ) (A \vec{v}),
  \end{equation}
  where the averaging operator $A$ defined above is of the same order of accuracy
  as the derivative operator $D$, i.e. it fulfils $(A \vec{u})_i = u(x_i) + \O(h^2)$
  in the interior and $(A \vec{u})_{0,N} = u(x_{0,N}) + \O(h)$ at the boundaries
  for a smooth function $u$.
\end{lemma}

Similarly, a general chain rule as discrete analogue of \eqref{eq:BV-chain-rule-step-function}
is satisfied. Indeed, if $f$ is continuously differentiable and $\vec{u}$ a possibly
vector valued grid function,
\begin{equation}
\begin{aligned}
  \bigl( D f(\vec{u}) \bigr)_i
  &=
  \frac{f(\vec{u}_{i+1}) - f(\vec{u}_{i-1})}{2h}
  \\
  &=
  \underbrace{\int_0^1 f'\bigl( \vec{u}_{i-1} + s (\vec{u}_{i+1} - \vec{u}_{i-1}) \bigr) \dif s}_{=:(A_{f'} \vec{u})_i}
  \cdot \frac{\vec{u}_{i+1} - \vec{u}_{i-1}}{2h}
  =
  (A_{f'} \vec{u})_i \cdot (D \vec{u})_i
\end{aligned}
\end{equation}
for interior nodes, where the possibly nonlinear averaging operator $A_{f'}$ has
been introduced. At the boundary nodes, it is given by
\begin{equation}
\begin{aligned}
  (A_{f'} \vec{u})_0
  &=
  \int_0^1 f'\bigl( \vec{u}_{0} + s (\vec{u}_{1} - \vec{u}_{0}) \bigr) \dif s
  \approx
  f'\bigl( u(x_0) \bigr),
  \\
  (A_{f'} \vec{u})_N
  &=
  \int_0^1 f'\bigl( \vec{u}_{N-1} + s (\vec{u}_{N} - \vec{u}_{N-1}) \bigr) \dif s
  \approx
  f'\bigl( u(x_N) \bigr).
\end{aligned}
\end{equation}
This is summed up in
\begin{lemma}
\label{lem:chain-rule}
  The classical second order derivative operator $D$ (on a periodic grid or with
  boundary closures) satisfies the chain rule
  \begin{equation}
    D f(\vec{u}) = (A_{f'} \vec{u}) \cdot (D \vec{u}),
  \end{equation}
  where the averaging operator $A_{f'}$ defined above is of the same order of
  accuracy as the derivative operator $D$, i.e. it fulfils
  $(A_{f'} \vec{u})_i = f'\bigl( u(x_i) \bigr) + \O(h^2)$
  in the interior and
  $(A_{f'} \vec{u})_{0,N} = f'\bigl( u(x_{0,N}) \bigr) + \O(h)$
  at the boundaries for smooth (and possibly vector valued) functions $u$ and $f$.
\end{lemma}

\begin{remark}
\label{rem:A-Af'}
  The averaging operator $A$ used for the product rule is a special case of the
  general averaging operator $A_{f'}$. Indeed, $A = A_{\id}$, where $\id$ is the
  identity mapping.
\end{remark}

\begin{remark}
  In general, $A_{f'}$ is neither a linear operator nor an averaging operator
  acting on $f'(\vec{u})$. Instead, it is a possibly nonlinear
  operator that uses intermediate values of $\vec{u}$ to average $f'$. It is
  linear if and only if $f'$ is linear, in particular
  in the case $f' = \id$, i.e. $A_{\id} = A$ discussed in Remark~\ref{rem:A-Af'}.
\end{remark}

\subsection{Higher Order Derivative Operators}

The product and chain rules for second order derivative operators cannot be
generalised to higher order derivative operators. In order to prove this, the
asymptotic expansion of the error of the derivative operator will be used.
\begin{lemma}
\label{lem:asymptotic-expansion-D1}
  Assume that $D$ is a discrete derivative operator of order $p$, i.e.
  $(D \vec{u})_i = u'(x_i) + \O(h^p)$ or, equivalently, $D$ is exact for polynomials
  of degree $\leq p$, with $p$ maximal. If $u$ is a smooth scalar-valued function,
  \begin{equation}
    (D \vec{u})_i = u'(x_i) + u^{(p+1)}(x_i) C^D_i h^p + \O(h^{p+1}),
  \end{equation}
  where $C_i^D h^p = \O(h^p)$ depends only on the grid and the derivative operator.
\end{lemma}
\begin{proof}
  By Taylor expansion, using the exactness of $D$ for polynomials of degree
  $\leq p$,
  \begin{equation}
  \begin{aligned}
    &\phantom{=\;}
    (D \vec{u})_i
    =
    \sum_j D_{ij} \vec{u}_j
    =
    \sum_j D_{ij} u(x_j)
    \\
    &=
    \sum_j D_{ij} \Bigl(
      u(x_i)
      + u'(x_i) (x_j - x_i)
      + \dots
      + \frac{1}{(p+1)!} u^{(p+1)}(x_i) (x_j - x_i)^{p+1}
      + \O(h^{p+2})
    \Bigr)
    \\
    &=
    u'(x_i)
    + u^{(p+1)}(x_i) \underbrace{\sum_j \frac{1}{(p+1)!} D_{ij} (x_j - x_i)^{p+1}}_{=: C^D_i h^p}
    + \O(h^{p+1}).
  \end{aligned}
  \end{equation}
  Here, $C^D_i h^p = \O(h^p)$, since $D$ scales as $h^{-1}$.
\QED
\end{proof}

This can be used to prove one of the main observations of this article.
\begin{theorem}
\label{thm:product-rule}
  If $D$ is a discrete derivative operator of order $p > 2$, there can be no
  averaging operator $A$ of order $q \in \N$ such that there is a product rule of
  the form $D (\vec{u} \vec{v}) = (A \vec{u}) (D \vec{v}) + (D \vec{u}) (A \vec{v})$.
\end{theorem}
\begin{proof}
  Consider the asymptotic expansions
  \begin{equation}
  \label{eq:expansion-Duv}
    \bigl( D (\vec{u} \vec{v}) \bigr)_i
    =
    (u v)'(x_i) + (u v)^{(p+1)}(x_i) C^D_i h^p + \O(h^{p+1})
  \end{equation}
  and
  \begin{equation}
  \begin{aligned}
    (A \vec{u})_i (D \vec{v})_i
    &=
    \left( u(x_i) + C^A_i(u) h^q + \O(h^{q+1}) \right)
    \left( v'(x_i) + v^{(p+1)}(x_i) C^D_i h^p + \O(h^{p+1}) \right),
    \\
    (D \vec{u})_i (A \vec{v})_i
    &=
    \left( u'(x_i) + u^{(p+1)}(x_i) C^D_i h^p + \O(h^{p+1}) \right)
    \left( v(x_i) + C^A_i(v) h^q + \O(h^{q+1}) \right),
  \end{aligned}
  \end{equation}
  where $C^A_i(u)$ is the leading order coefficient for $A$ and may depend on the
  function $u$ and its derivatives. There are three different cases: $q < p$,
  $q = p$, and $q > p$.

  If $q < p$, the product rule cannot hold, because the terms
  \begin{equation}
    \left(
      u'(x_i) C^A_i(v) + v'(x_i) C^A_i(u)
    \right) h^q
  \end{equation}
  involving $h^q$ are not matched by terms in \eqref{eq:expansion-Duv}. Basically,
  $D (\vec{u} \vec{v})$ is a $p$-th order approximation while
  $(A \vec{u}) (D \vec{v}) + (D \vec{u}) (A \vec{v})$ is only a $q$-th order
  approximation.

  If $q = p$, the product rule can only hold if the terms involving $h^p = h^q$
  are equal, i.e. if
  \begin{multline}
  \label{eq:proof-1}
    (u v)^{(p+1)}(x_i) C^D_i
    \\
    =
    \left( u(x_i) v^{(p+1)}(x_i) + u^{(p+1)}(x_i) v(x_i) \right) C^D_i
    + u'(x_i) C^A_i(v) + v'(x_i) C^A_i(u).
  \end{multline}
  Since
  \begin{equation}
    (u v)^{(p+1)}(x_i)
    =
    \sum_{k=0}^{p+1} \binom{p+1}{k} u^{(k)}(x_i) v^{(p+1-k)}(x_i),
  \end{equation}
  the terms with $k=0$ and $k=p+1$ match the braces on the right hand side of
  \eqref{eq:proof-1}, but the remaining terms can only match if $p \leq 2$,
  since the remaining sum cannot be factored as on the right hand side.

  Finally, if $q > p$, the terms involving $h^p$ do not match, because
  \begin{equation}
    (u v)^{(p+1)}(x_i)
    \neq
    u(x_i) v^{(p+1)}(x_i) + u^{(p+1)}(x_i) v(x_i)
  \end{equation}
  for $p \geq 1$ in general.
\QED
\end{proof}

\begin{remark}
  Using polynomial collocation methods on Lobatto Legendre or Gauss Legendre
  nodes in $[-1,1]$, a discrete product rule holds for $p = 1$, i.e. for two
  nodes, since they are of the same form as the classical finite difference
  derivative operator. However, for $p=2$, there can be no product rule. Indeed,
  for Lobatto nodes $\set{-1,0,1}$ and $u(x) = (1+x)^2 = v(x)$, the discrete
  derivatives of $u$ and $v$ at $-1$ are zero (since they are exact), but the
  discrete derivative of $uv$ at $-1$ is
  \begin{equation}
    (D \vec{u} \vec{v})_{-1}
    =
    - \frac{3}{2} \vec{u}_{-1} \vec{v}_{-1}
    + 2 \vec{u}_0 \vec{v}_0
    - \frac{1}{2} \vec{u}_1 \vec{v}_1
    =
    0 + 2 \cdot 1^2 - \frac{1}{2} \cdot 4^2
    =
    - 6
    \neq 0.
  \end{equation}
  A similar argument holds for Gauss Legendre nodes.
\end{remark}

Since the product rule is a special case of the chain rule with vector valued
functions $u$ (cf. Remark~\ref{rem:CR-as-PR}), a general chain rule is also
excluded for discrete derivative operators of higher order of accuracy. However,
this argument does not forbid a chain rule for scalar valued functions. Nevertheless,
this case is also excluded by the second main observation of this article.
\begin{theorem}
\label{thm:chain-rule}
  If $D$ is a discrete derivative operator of order $p > 2$, there can be no
  general averaging operator $A_{f'}$ of order $q \in \N$ such that there is a
  chain rule of the form $D\bigl( f(\vec{u}) \bigr) = (A_{f'} \vec{u}) \cdot (D \vec{u})$.
\end{theorem}
\begin{proof}
  By the argument above, it suffices to consider scalar valued functions. In this
  case,
  \begin{equation}
    \bigl(D f(\vec{u}) \bigr)_i
    =
    f'(\vec{u}_i) u'(x_i)
    + \bigl( f(u) \bigr)^{(p+1)}(x_i) C^D_i h^p
    + \O(h^{p+1})
  \end{equation}
  and
  \begin{multline}
    (A_{f'} \vec{u})_i (D \vec{u})_i
    =
    \left( f'(\vec{u}_i) + C^A_i\bigl(f'(u)\bigr) h^q + \O(h^{q+1}) \right)
    \\
    \left( u'(x_i) + u^{(p+1)}(x_i) C^D_i h^p + \O(h^{p+1}) \right).
  \end{multline}
  Again, there are three different cases: $q < p$, $q = p$, and $q > p$.

  If $q < p$, the chain rule cannot hold, because $D\bigl( f(\vec{u}) \bigr)$ is
  a $p$-th order approximation while $(A_{f'} \vec{u}) (D u)$ is only a $q$-th
  order approximation.

  If $q = p$, $\bigl( f(u) \bigr)^{(p+1)}(x_i)$ can be expressed using the formula
  of Faà di Bruno \cite[Lemma~II.2.8, simplified for the scalar case]{hairer2008solving},
  \begin{equation}
    \bigl( f(u) \bigr)^{(p+1)}(x_i)
    =
    \sum_{u \in LS_{p+2}} f^{(m)}(\vec{u}_i)
      \, u^{(\delta_1)}(x_i) \dots u^{(\delta_m)}(x_i).
  \end{equation}
  Here, $LS_{p+2}$ is the set of special labelled trees of order $p+2$ which have
  no ramifications except at the root, $m$ is the number of branches leaving the
  root, and $\delta_1, \dots, \delta_m$ are the numbers of nodes in each of these
  branches, see \cite[Lemma~II.2.8]{hairer2008solving}.
  Thus, it is clear that $u'(x_i)$ cannot be factored out of the remaining terms
  after subtracting $f'(\vec{u}_i) u^{(p+1)}(x_i)$ if $p > 2$.

  Finally, if $q > p$, the terms involving $h^p$ do not match, because
  \begin{equation}
    \bigl( f(u) \bigr)^{(p+1)}(x_i)
    \neq
    f'(\vec{u}_i) u^{(p+1)}(x_i)
  \end{equation}
  for $p \geq 1$ in general.
\QED
\end{proof}

\begin{remark}
  A product rule for classical difference operators with error term of the form
  \begin{equation}
    (D \vec{u} \vec{v})_i
    =
    \vec{u}_i (D \vec{v})_i + (\partial_x u)_i (A \vec{v})_i + e_i
  \end{equation}
  has been used in \cite[Lemma~3.1 and Lemma~3.2]{mishra2010stability}. If $u$ is
  smooth, $(\partial_x u)_i$ is the derivative at $x_i$ and $\norm{e} \leq C h \norm{\vec{v}}$
  for some constant $C > 0$. The averaging operator $A$ is linear and of the same
  order of accuracy as the derivative operator $D$.
\end{remark}

\begin{remark}
  The investigation of discrete product and chain rules is also somewhat loosely
  related to the entropy stability and conservation theory initiated by Tadmor
  \cite{tadmor1987numerical,tadmor2003entropy}. Indeed, instead of a chain rule
  of the form $\partial_x f(u) = \widehat{f'(u)} \partial_x u$ \eqref{eq:BV-chain-rule},
  a discrete version of
  $U'(u) \cdot \widetilde{\partial_x f(u)} = \widetilde{\partial_x F(u)}$
  is used, where $U$ is the entropy fulfilling $U'(u) \cdot f'(u) = F'(u)$. Such
  approximations can be found for arbitrary order, cf.
  \cite{lefloch2002fully,sjogreen2010skew,fisher2013high,ranocha2017comparison,chen2017entropy}.
  Basically, schemes of lower order can be extrapolated if regular grids are used,
  cf. \cite[Section~3.2]{ranocha2018thesis}. Nevertheless, they can be used also on
  certain irregular grids.
\end{remark}

\section{Entropy Stability of Discrete Second Derivatives}
\label{sec:laplace}

In order to regularise a hyperbolic conservation law $\partial_t u + \partial_x f(u) = 0$,
where $u$ are the conserved variables and $f(u)$ is the flux, a parabolic term
can be added to the right-hand side, resulting in
\begin{equation}
  \partial_t u(t,x) + \partial_x f\bigl( u(t,x) \bigr)
  =
  \partial_x \bigl( \epsilon(x) \partial_x u(t,x) \bigr),
\end{equation}
where $\epsilon \geq 0$ controls the amount of viscosity. An entropy is a convex
function $U$ satisfying $U'(u) \cdot f'(u) = F'(u)$, where $F$ is the corresponding
entropy flux. Thus, smooth solutions of the conservation law fulfil the additional
conservation law
\begin{equation}
  \partial_t U(u)
  =
  U'(u) \cdot \partial_t u
  =
  - U'(u) \cdot f'(u) \cdot \partial_x u
  =
  - \partial_x F(u)
\end{equation}
and an entropy inequality $\partial_t U + \partial_x F \leq 0$ is required for
weak solutions, cf. \cite[Chapter~IV]{dafermos2010hyperbolic}. The viscosity
term on the right-hand side induces a global entropy inequality for sufficiently
smooth solutions. Indeed, in a periodic domain $\Omega$,
\begin{equation}
\label{eq:smooth-entropy-dissipation-CR}
  \int_\Omega U' \cdot \partial_x (\epsilon \partial_x u) \dif x
  =
  - \int_\Omega \epsilon (\partial_x U') \cdot \partial_x u \dif x
  =
  - \int_\Omega \epsilon (\partial_x u) \cdot U'' \cdot \partial_x u \dif x
  \leq
  0,
\end{equation}
since $U$ is convex and $\epsilon \geq 0$. In a non-periodic domain $\Omega$,
if $\epsilon$ vanishes on $\partial\Omega$, the same result holds. Otherwise,
there will be additional boundary terms.

The computation in \eqref{eq:smooth-entropy-dissipation-CR}
relies on the chain rule. Thus, it might be conjectured that second order difference
approximations of the Laplace operator (with possibly varying coefficients) are
also dissipative for every entropy $U$ and that higher order difference
approximations to the second derivative are not necessarily dissipative for every
entropy $U$.

\subsection{Second Order Derivative Operators}

In a periodic domain, the classical second order difference approximation to the
Laplace operator is given by
\begin{equation}
  (D_2 \vec{u})_i
  =
  \frac{\vec{u}_{i+1} - 2 \vec{u}_i + \vec{u}_{i-1}}{h^2}.
\end{equation}
Thus, multiplying pointwise by $U'(\vec{u}_i) = U'_i$ and summing up all terms yields
due to the periodicity of the domain
\begin{equation}
\begin{aligned}
  h^2 \sum_i U'_i \cdot (D_2 \vec{u})_i
  &=
  \sum_i U'_i \cdot (\vec{u}_{i+1} - \vec{u}_i) - \sum_i U'_i \cdot (\vec{u}_i - \vec{u}_{i-1})
  \\
  &=
  -\sum_i (U'_{i+1} - U'_i) \cdot (\vec{u}_{i+1} - \vec{u}_i)
  \leq
  0,
\end{aligned}
\end{equation}
since $U'$ is monotone. If $u$ is scalar valued, $U'$
is monotonically increasing, since $U'' \geq 0$ due to the convexity of $U$. If
$u$ is vector valued, the usual generalised definition of monotonicity is used,
i.e. $(U'(u) - U'(v)) \cdot (u - v) \geq 0$,
cf. \cite[section~II.2, p.~37]{showalter1997monotone}. Indeed, due to the
convexity of $U$,
\begin{multline}
  (\vec{u}_{i+1} - \vec{u}_i) \cdot \bigl( U'(\vec{u}_{i+1}) - U'(\vec{u}_i) \bigr)
  \\
  =
  \int_0^1 (\vec{u}_{i+1} - \vec{u}_i)
  \cdot U''\bigl( \vec{u}_i + s (\vec{u}_{i+1} - u_i) \bigr)
  \cdot (\vec{u}_{i+1} - \vec{u}_i) \dif s
  \geq
  0.
\end{multline}
This is exactly the chain rule for classical difference approximations.

If a variable coefficient $\epsilon \geq 0$ is considered in a periodic domain,
a second order approximation to $\partial_x (\epsilon \partial_x u)$ is given by
\begin{equation}
\label{eq:SBP-var-coef-2-interior}
  (D_2^\epsilon \vec{u})_i
  =
  \frac{\epsilon_{i} + \epsilon_{i+1}}{2 h^2} \vec{u}_{i+1}
  - \frac{\epsilon_{i-1} + 2 \epsilon_{i} + \epsilon_{i+1}}{2 h^2} \vec{u}_i
  + \frac{\epsilon_{i-1} + \epsilon_{i}}{2 h^2} \vec{u}_{i-1},
\end{equation}
cf. \cite{mattsson2012summation}. Using again the periodicity and the convexity
of $U$,
\begin{equation}
\begin{aligned}
  h^2 \sum_i U'_i \cdot (D_2^\epsilon \vec{u})_i
  &=
  \sum_i \frac{\epsilon_{i} + \epsilon_{i+1}}{2} U'_i \cdot (\vec{u}_{i+1} - \vec{u}_i)
  - \sum_i \frac{\epsilon_{i-1} + \epsilon_{i}}{2} U'_i \cdot (\vec{u}_i - \vec{u}_{i-1})
  \\
  &=
  -\sum_i
  \frac{\epsilon_{i} + \epsilon_{i+1}}{2} (U'_{i+1} - U'_i) \cdot (\vec{u}_{i+1} - \vec{u}_i)
  \leq
  0.
\end{aligned}
\end{equation}

Summation-by-parts operators for second derivatives with variable coefficients
have been developed in \cite{mattsson2012summation}. The second order discrete
derivative in the interior is given by \eqref{eq:SBP-var-coef-2-interior} and
equipped with the boundary closures
\begin{equation}
\begin{aligned}
  (D_2^\epsilon \vec{u})_0
  &=
  (2 \epsilon_0 - \epsilon_1) \vec{u}_0
  + (-3 \epsilon_0 + \epsilon_1) \vec{u}_1
  + \epsilon_0 \vec{u}_3,
  \\
  (D_2^\epsilon \vec{u})_N
  &=
  (2 \epsilon_N - \epsilon_{N-1}) \vec{u}_N
  + (-3 \epsilon_N + \epsilon_{N-1}) \vec{u}_{N-1}
  + \epsilon_N \vec{u}_{N-2}.
\end{aligned}
\end{equation}
If the variable coefficient $\epsilon$ vanishes at the boundary, i.e. if
$\epsilon_0 = 0 = \epsilon_N$, these boundary closures become
\begin{equation}
  (D_2^\epsilon \vec{u})_0
  =
  \epsilon_1 (\vec{u}_1 - \vec{u}_0),
  \qquad
  (D_2^\epsilon \vec{u})_N
  =
  -\epsilon_{N-1} (\vec{u}_N - \vec{u}_{N-1}).
\end{equation}
Since the discrete integral is given as a quadrature with weights on the diagonal
of the mass/norm matrix $H = \diag{1/2, 1, \dots, 1, 1/2}$, the discrete equivalent
of the integral $\int_\Omega U' \cdot \partial_x(\epsilon \partial_x u)$ is
\begin{equation}
\begin{aligned}
  &\phantom{=\;}
  \sum_{i=0}^N H_{ii} U'_i \cdot (D_2^\epsilon \vec{u})_i
  \\
  &=
    \frac{1}{2} \epsilon_1 U'_0 \cdot (\vec{u}_1 - \vec{u}_0)
  - \frac{1}{2} \epsilon_{N-1} U'_N \cdot (\vec{u}_N - \vec{u}_{N-1})
  \\&\quad
  + \sum_{i=1}^{N-1} \frac{\epsilon_{i} + \epsilon_{i+1}}{2} U'_i \cdot (\vec{u}_{i+1} - \vec{u}_{i})
  - \sum_{i=1}^{N-1} \frac{\epsilon_{i-1} + \epsilon_{i}}{2} U'_i \cdot (\vec{u}_{i} - \vec{u}_{i-1})
  \\
  &=
    \frac{\epsilon_0 + \epsilon_1}{2} U'_0 \cdot (\vec{u}_1 - \vec{u}_0)
  - \frac{\epsilon_{N-1} + \epsilon_N}{2} U'_N \cdot (\vec{u}_N - \vec{u}_{N-1})
  \\&\quad
  + \sum_{i=1}^{N-1} \frac{\epsilon_{i} + \epsilon_{i+1}}{2} U'_i \cdot (\vec{u}_{i+1} - \vec{u}_{i})
  - \sum_{i=0}^{N-2} \frac{\epsilon_{i} + \epsilon_{i+1}}{2} U'_{i+1} \cdot (\vec{u}_{i+1} - \vec{u}_{i})
  \\
  &=
  - \sum_{i=0}^N \frac{\epsilon_{i} + \epsilon_{i+1}}{2} (U'_{i+1} - U'_i) \cdot (\vec{u}_{i+1} - \vec{u}_{i})
  \\
  &\leq
  0.
\end{aligned}
\end{equation}
This proves
\begin{theorem}
\label{thm:laplace-2nd-order}
  The discretisations $D_2^\epsilon$ of the second derivative operator
  $\partial_x (\epsilon \partial_x \cdot)$
  with possibly varying coefficients $\epsilon \geq 0$ given above in periodic
  domains or on bounded domains with $\epsilon_0 = 0 = \epsilon_N$ are entropy
  dissipative for every convex entropy.
\end{theorem}

\begin{remark}
\label{rem:laplace-2nd-order}
  The statement of Theorem~\ref{thm:laplace-2nd-order} holds for the second order
  summation-by-parts operator $D_2^\epsilon$ (and its interior stencil in periodic
  domains) mentioned above. It is not necessarily true for every second order
  approximation of $\partial_x (\epsilon \partial_x \cdot)$. Indeed, in a periodic
  domain, such an approximation is also given by
  \begin{equation}
    (\widetilde D_2^\epsilon \vec{u})_i
    =
    \epsilon_i \frac{\vec{u}_{i+1} - 2 \vec{u}_{i} + \vec{u}_{i-1}}{h^2}
    + \frac{\epsilon_{i+1} - \epsilon_{i-1}}{2 h} \frac{\vec{u}_{i+1} - \vec{u}_{i-1}}{2 h}.
  \end{equation}
  Choose the grid $x_i = i$, $i \in \set{0,1,2,3}$, with periodic boundary conditions,
  i.e. $\vec{u}_3 = \vec{u}_0$. Set $\vec{u} = (\vec{u}_0,\vec{u}_1,\vec{u}_2)
  = (0.6, 0.8, 0.2)$, $\epsilon = (0.4, 0.2, 0.8)$ and use the entropy given by
  $U(u) = u$. Then, $U'(u) = 1$ and
  \begin{equation}
    \sum_{i=0}^2 U'(\vec{u}_i) \cdot (\widetilde D_2^\epsilon \vec{u})_i
    =
    -0.17 - 0.20 + 0.79
    =
    0.42
    >
    0.
  \end{equation}
  While this does not prove that the SBP operator mentioned above is the only
  second order entropy dissipative approximation, it illustrates the good
  properties of this operator.
\end{remark}

\subsection{Higher Order Derivative Operators}

Since there is no discrete chain rule for higher order difference approximations
to the first derivative, it might be conjectured that discrete higher order
second derivatives are in general not entropy dissipative. In order to prove
this, it suffices to consider the case of constant coefficients. At the grid point
$x_j$, a general (linear) discrete approximation of the second derivative can be
written as
\begin{equation}
\label{eq:2nd-derivative-coefficients}
  (D_2 \vec{u})_j = \sum_k c_k \vec{u}_{j+k}.
\end{equation}
The following result will be used.
\begin{lemma}
\label{lem:2nd-derivative-coefficients}
  If \eqref{eq:2nd-derivative-coefficients} is an approximation of the second
  derivative with order of accuracy $p > 2$, there is a $k \neq 0$ such that
  $c_k < 0$.
\end{lemma}
\begin{proof}
  Using Taylor expansion, the order conditions for an order of accuracy $p = 3$ are
  \begin{equation}
  \begin{aligned}
    \sum_k c_k &= 0,
    \qquad &
    \sum_k c_k (x_{j+k} - x_j) &= 0,
    \qquad &
    \sum_k c_k (x_{j+k} - x_j)^2 &= 2,
    \\
    \sum_k c_k (x_{j+k} - x_j)^3 &= 0,
    \qquad &
    \sum_k c_k (x_{j+k} - x_j)^4 &= 0.
  \end{aligned}
  \end{equation}
  Due to the last condition, at least one $c_k$ with $k \neq 0$ must be negative.
\QED
\end{proof}
\begin{example}
  The classical fourth order approximation to the second derivative on a periodic
  domain is given by
  \begin{equation}
    h^2 (D_2 \vec{u})_i
    =
    - \frac{1}{12} (\vec{u}_{i+2} + \vec{u}_{i-2})
    + \frac{4}{3} (\vec{u}_{i+1} + \vec{u}_{i-1})
    - \frac{5}{2} \vec{u}_i.
  \end{equation}
\end{example}

Lemma~\ref{lem:2nd-derivative-coefficients} can be used to prove the last main
observation of this article.
\begin{theorem}
\label{thm:laplace-higher-order}
  If $D_2$ is a discrete derivative operator approximating the second derivative
  with order of accuracy $p > 2$, it is not dissipative for every entropy.
\end{theorem}
\begin{proof}
  Consider the grid point $x_j$. Writing the approximation of the second derivative
  as in \eqref{eq:2nd-derivative-coefficients}, there is some coefficient $c_k < 0$,
  $k \neq 0$, due to Lemma~\ref{lem:2nd-derivative-coefficients}. Fix $\vec{u}_{j+k} < 0$,
  say $\vec{u}_{j+k} = -1$. Set $\vec{u}_j = \epsilon > 0$, where $\epsilon > 0$
  will be fixed later. Choose $\vec{u}_i = 0$ for $i \neq j, j+k$. Finally,
  consider the entropy $U(u) = \max\set{0, u - \epsilon/2}$. Then, $U'(u) = 0$
  for $u < \epsilon/2$ and $U'(u) = 1$ for $u > \epsilon/2$. Thus,
  \begin{equation}
    \sum_i U'(\vec{u}_i) \cdot (D_2 \vec{u})_i
    =
    U'(\vec{u}_j) \cdot (D_2 \vec{u})_j
    =
    1 \cdot \sum_l c_l \vec{u}_{j+l}
    =
    c_0 \underbrace{\vec{u}_j}_{= \epsilon}
    + \underbrace{c_k}_{< 0} \underbrace{\vec{u}_{j+k}}_{< 0}
    >
    0,
  \end{equation}
  if $\epsilon > 0$ is chosen small enough.
\QED
\end{proof}

\begin{remark}
  The entropy $U$ used in the proof of Theorem~\ref{thm:laplace-higher-order} can
  be made smooth by suitable modifications around $u = \epsilon/2$.
\end{remark}

\begin{remark}
  Of course, higher order approximations to second derivatives that are dissipative
  for a specific entropy can be constructed. Classical difference operators are
  negative semidefinite, i.e. they are dissipative for the $L^2$ entropy
  $U(u) = \frac{1}{2} u^2$ with $U'(u) = u$. For  a general entropy $U$, entropy
  dissipative second derivatives can be constructed by using the entropy variables
  $w := U'(u)$ instead of the conserved variables $u$, cf. \cite{fisher2013high}.
\end{remark}

\begin{remark}
  In periodic domains, the classical central finite difference approximations
  to the first derivative of higher order can be constructed via extrapolation
  from the second order operator, cf. \cite[Section~3.2]{ranocha2018thesis}.
  Thus, by enforcing positivity of the corresponding coefficients for the second
  derivative, entropy dissipative terms can be constructed similarly for higher
  order first derivative operators, as used in \cite{svard2009shock}. However,
  these are not higher order approximations of the second derivative.
\end{remark}

\section{Summary and Discussion}
\label{sec:summary}

In this article, product and chain rules using averaged compositions have been
shown to hold for second order approximations to first order derivative operators,
similarly to corresponding results for functions of bounded variation
(Lemma~\ref{lem:product-rule} and Lemma~\ref{lem:chain-rule}). While such mimetic
properties may have nice implications, it is proven that such results cannot hold
for higher order approximations, independently of the grid or the exact form
of the discrete derivative operator (Theorem~\ref{thm:product-rule} and
Theorem~\ref{thm:chain-rule}). This result holds also for spectral
collocation and nodal discontinuous Galerkin methods.

Furthermore, the entropy dissipation induced by difference operators approximating
second derivatives with varying coefficients is studied. While certain second order
approximations are dissipative for all entropies (Theorem~\ref{thm:laplace-2nd-order}),
such a result is not valid for higher order approximations
(Theorem~\ref{thm:laplace-higher-order}).

These results (Theorems~\ref{thm:product-rule}, \ref{thm:chain-rule},
\ref{thm:laplace-higher-order}) have been proven for linear difference operators.
Indeed, they rely on Lemma~\ref{lem:asymptotic-expansion-D1} and
Lemma~\ref{lem:2nd-derivative-coefficients}, which assume linearity. Thus,
similar to classical results for (scalar) conservation laws \cite{harten1987nonlinearity},
it might be possible to construct nonlinear operators approximating the first
and second derivative such that desirable properties can be obtained.

While these results are interesting on their own, there are several connections with
other results and open questions. It is well known that higher order schemes can be
more efficient for certain problems than lower order ones \cite{kreiss1972comparison}.
However, the numerical treatment of discontinuities in solutions to hyperbolic
conservation laws has to be well-considered, especially for higher order schemes.
Even though a single entropy inequality can be sufficient for genuinely nonlinear
scalar conservation laws \cite{panov1994uniqueness,delellis2004minimal,krupa2017single}, general conservation laws pose additional challenges \cite{lefloch2002hyperbolic}.
Since certain mimetic properties discussed in this article are limited to second
order schemes, suitable detection of discontinuities and corresponding adaptations
of the numerical methods may be inevitable.

\section*{Acknowledgements}

This work was supported by the German Research Foundation (DFG, Deutsche
Forschungsgemeinschaft) under Grant SO~363/14-1. The author would like to thank
the anonymous reviewers for their helpful comments and valuable suggestions to
improve this article.

\appendix

\printbibliography

\end{document}